\documentclass[12pt, oneside, reqno]{article}
\usepackage{amsmath,amsthm}
\usepackage{amssymb,latexsym}
\usepackage{enumerate}
\usepackage{cases}

\newtheorem{thm}{Theorem}[section]

\newtheorem{lem}[thm]{Lemma}
\newtheorem{prob}[thm]{Problem}

\theoremstyle{definition}

\newtheorem{rem}[thm]{Remark}

\numberwithin{equation}{section}

\begin{document}

\title{\Large  The restricted sumsets in $\mathbb{Z}_n$}
\author{\large Min Tang\thanks{Corresponding author. This work was supported by
National Natural Science Foundation of China, Grant No. 11471017.
E-mail: tmzzz2000@163.com} and Meng-Ting Wei}
\date{} \maketitle
 \vskip -3cm
\begin{center}
\vskip -1cm { \small
\begin{center}
School of Mathematics and Statistics, Anhui Normal
University
\end{center}
\begin{center}
Wuhu 241002, PR China
\end{center}}
\end{center}

 {\bf Abstract:} Let $h\geq 2$ be a positive integer. For any subset $\mathcal{A}\subset \mathbb{Z}_n$, let $h^{\wedge}\mathcal{A}$ be the set of the elements of $\mathbb{Z}_n$ which are sums of $h$ distinct elements of $\mathcal{A}$.
 In this paper, we obtain some new results on $4^{\wedge}\mathcal{A}$ and $5^{\wedge}\mathcal{A}$. For example, we show that if $|\mathcal{A}|\geq 0.4045n$ and $n$ is odd, then $4^{\wedge}\mathcal{A}=\mathbb{Z}_{n}$;
Under some conditions, if $n$ is even and $|\mathcal{A}|$ is close to $n/4$, then $4^{\wedge}\mathcal{A}=\mathbb{Z}_{n}$.


\vskip8mm

\section{Introduction} Let $p$ be an odd prime and $h\geq 2$ be a positive integer. For any subset $\mathcal{A}\subset \mathbb{Z}_n$, let $h^{\wedge}\mathcal{A}$ be the set of the elements of $\mathbb{Z}_n$ which are sums of $h$ distinct elements of $\mathcal{A}$.
Up to now, we know quite few on restricted sums in $\mathbb{Z}_n$. In 1963, Erd\H{o}s-Heilbronn first posed the following conjecture at a number theory conference:

{\bf Erd\H{o}s-Heilbronn Conjecture} If $\mathcal{A}\subseteq \mathbb{Z}_p$ and $|\mathcal{A}|=k$, then $$|h^{\wedge}A|\geq \min\{p, hk-h^{2}+1\}.$$

In 1994, Dias da Silve and Hamidoune \cite{Dias} proved this conjecture by using the knowledge of representation theory and linear algebra. Let $p$ be the characteristic of field $F$. They proved the following result:

\noindent{\bf Theorem A}(\cite{Dias}, Theorem 4.1) Let $\mathcal{A}$ be a finite subset of a field $F$ and $m$ be a positive integer. Then
$$|m^{\wedge}\mathcal{A}|\geq \min\{p, m|\mathcal{A}|-m^2+1\}.$$

Let $\mathcal{A}\subseteq \mathbb{Z}_n$. Theorem A implies the following fact: If $n$ is prime and $|\mathcal{A}|>(n+8)/3$, then $3^{\wedge}\mathcal{A}=\mathbb{Z}_n$.

In 1995, Alon, Nathanson and Ruzsa \cite{Alon1995} proved Erd\H{o}s-Heilbronn conjecture for $h=2$.
In 1996, by general algebraic technique, Alon, Nathanson and Ruzsa \cite{Alon1996} proved this conjecture for all $h\geq 2$. In 1999, Alon \cite{Alon1999} restated this result by polynomial method.

In 1999, Gallardo, Grekos and Pihko \cite{LGJ} obtained the following result:

\noindent{\bf Theorem B}(\cite{LGJ}, Lemma 3) Let $\mathcal{A}\subseteq \mathbb{Z}_n$ such that $|\mathcal{A}|> n/2+1$, then $2^{\wedge}A=\mathbb{Z}_n$.

In 2002, Gallardo, Grekos, Habsieger, Hennecart, Landreau and Plagne \cite{Gallardo02} gave some results on restricted sums $2^{\wedge}A$ and $3^{\wedge}A$.

\noindent{\bf Theorem C} (\cite{Gallardo02}, Theorem 2.9) Let $\mathcal{A}$
be a subset of $\mathbb{Z}_{n}$ such that $\displaystyle|\mathcal{A}|=\left\lfloor\frac{n}{2}\right\rfloor+1$ and $|2^{\wedge}\mathcal{A}|=n-2$. We denote by $2a,2b\in2\mathcal{A}\setminus2^{\wedge}\mathcal{A}$, and
by $H$ the subgroup generated by $2(b-a)$, whose order is $d>1$.

Then
$d$ is odd and there exists a sequence $\varepsilon_{1},\varepsilon_{2},\cdots,\varepsilon_{m}\in\{\pm1\}$ such that
\begin{equation}\label{eq1000}\mathcal{A}=a+\bigcup\limits_{j=1}^{m}(H+j\varepsilon_{j})\cup\mathcal{B},\;m=\left\lfloor\frac{\frac{n}{d}-1}{2}\right\rfloor,\end{equation}
where
\begin{displaymath}
\mathcal{B}=
\begin{cases}
\{b,3b,5b,\cdots,-2b,0=b+\frac{d-1}{2}\cdot 2b\}, & \text{if }n\; \text{is odd},\\
\{b,3b,5b,\cdots,db\}\cup\{0,-2b,-4b,\cdots,-(d-1)b\},\quad& \text{if }n\; \text{is even}.
\end{cases}
\end{displaymath}

\noindent{\bf Theorem D} (\cite{Gallardo02}, Theorem 3.1) For any integer $n\geq 12$, except $n=15,$ and for any subset $\mathcal{A}$ of $\subset\mathbb{Z}_{n}$ such that $|\mathcal{A}|>n/2$, one has $3^{\wedge}\mathcal{A}=\mathbb{Z}_{n}$.

For other related problems, see (\cite{Dierker}, \cite{Lev01}, \cite{Lev02}, \cite{Gallardo}).

In this paper, we obtain the following results:

\begin{thm} {\label{thm2}}For any even positive integer $n\geq 6$, if $\mathcal{A}\subset\mathbb{Z}_{n}$ such that $|\mathcal{A}|\geq\displaystyle\frac{n}{2}+3$, then we have
$4^{\wedge}\mathcal{A}=5^{\wedge}\mathcal{A}=\mathbb{Z}_{n}.$
\end{thm}

\begin{thm} {\label{thm3}} Let $E$ and $O$ be the even and the odd elements of $\mathbb{Z}_n$, respectively. Write
$\mathcal{A}_{e}=E\cap\mathcal{A}$, $\mathcal{A}_{o}=O\cap\mathcal{A}.$ For any even positive integer $n$, if $\mathcal{A}\subset\mathbb{Z}_{n}$ such that $|\mathcal{A}_e|\geq\displaystyle\frac{n}{4}+3$ and $|\mathcal{A}_o|\geq2$, then we have
$4^{\wedge}\mathcal{A}=\mathbb{Z}_{n}.$
\end{thm}

\begin{rem} By Theorem A we know that if $n$ is prime and $|\mathcal{A}|>(n+15)/4$, then $4^{\wedge}\mathcal{A}=\mathbb{Z}_n$. Theorem \ref{thm3} shows that under some conditions, if $|\mathcal{A}|\geq \displaystyle\frac{n+20}{4}$, then for even integer $n$ we have $4^{\wedge}\mathcal{A}=\mathbb{Z}_{n}$.
\end{rem}
\begin{thm} {\label{thm1}}For any $\alpha>\alpha_{0}=\displaystyle\sqrt[3]{\frac{27+\sqrt{741}}{486}}+\sqrt[3]{\frac{27-\sqrt{741}}{486}}$, there exists $N=N(\alpha)=\displaystyle\frac{54}{9\alpha^{3}+\alpha-1}$ such that for all
 $n>N$ and $n$ is odd,
 $\mathcal{A}\subset\mathbb{Z}_{n}$, if $|\mathcal{A}|\geq\alpha n$, then we have $4^{\wedge}\mathcal{A}=\mathbb{Z}_{n}$.
\end{thm}

By Theorem D, we know that if $n$ is even and $|\mathcal{A}|$ is close to $n/2$, then $3^{\wedge}\mathcal{A}=\mathbb{Z}_{n}$. But we know few when $|\mathcal{A}|$ is less than $n/2$.
 We pose the following problem:

\begin{prob} {\label{prob}} For any $\mathcal{A}\subset\mathbb{Z}_{n}$ such that $|\mathcal{A}|\geq4$, whether $|3^{\wedge}\mathcal{A}|\geq |\mathcal{A}|$ or not?
\end{prob}

Throughout this paper, define the doubling constant $L(G)$ to be the maximal number of doubles that coincide in group $G$. For any $u\in\mathbb{R}$, define $$S(u)=\sum\limits_{a\in\mathcal{A} }e(ua),$$
where $e(u)=\text{exp}(2\pi iu)$.

\section{Lemmas}
\begin{lem}(\cite{Gallardo02}, Lemma 2.4){\label{lem4}}  For any subset $\mathcal{A}$ of $\mathbb{Z}_{n}$ such that $|\mathcal{A}|\geq3$, one has $|2^{\wedge}\mathcal{A}|\geq |\mathcal{A}|$.
\end{lem}
\begin{lem}(\cite{Gallardo02}, Lemma 2.2){\label{lem5}}  Let $\mathcal{A}$ be a subset of an Abelian group $G$.
If $|\mathcal{A}|>(|G|+L(G)) / 2$, then $|2^{\wedge}\mathcal{A}|=G$.
\end{lem}
\begin{lem}{\label{lem6}} Let $n$ be an even integer and let $E$ be the even elements of $\mathbb{Z}_n$. Then
\begin{displaymath}
L(E)=
\begin{cases}
2, &n\equiv 0 \pmod 4,\\
1, &n\equiv 2 \pmod 4.
\end{cases}
\end{displaymath}
\end{lem}

\begin{proof}
If $n\equiv 2 \pmod 4$, then $$E=\left\{2i:\ i=0,1,\ldots,\frac{n}{2}-1\right\}.$$
For all $0\leq i\neq j\leq \displaystyle\frac{n}{2}-1,$ we have $4i\not\equiv4j\pmod n.$
In fact, if $4i\equiv 4j\pmod n$, then $2i\equiv 2j\pmod {n/2}$. Since $(2,n/2)=1$, we have $i\equiv j\pmod {n/2}$, contradiction.
Hence $L(E)=1$.

If $n\equiv 0 \pmod 4$, then
$$E=\left\{2i:\ i=0,1,\ldots,\frac{n}{4}-1\right\}\cup \left\{2i+\frac{n}{2}:\ i=0,1,\ldots,\frac{n}{4}-1\right\}.$$
Noting that $0\leq i\neq j\leq \displaystyle\frac{n}{4}-1,$ we have $4i\not\equiv4j\pmod n$, and
$$2\cdot 2i\equiv 2\cdot\left(2i+\frac{n}{2}\right)\pmod n.$$Hence $L(E)=2$.
\end{proof}

\begin{lem}(\cite{Gallardo02}, Lemma 3.3){\label{lem1}} Let $d\geq3$ be an odd integer, and $X$ be a positive real number. For any $\underline{x}=(x_{1},\ldots,x_{n})\in \mathbb{R}^{d}$, we put
$$T(\underline{x})=\sum\limits_{j=1}^{d}x_{j}e(j/d).$$
Then one has $$\max\limits_{\underline{x}\in[0,X]^{d}}|T(x)|=\displaystyle\frac{X}{2\sin(\pi/2d)}.$$
\end{lem}

\begin{lem}{\label{lem2}} For any $m\in\mathbb{Z}_{n}$, let $R(m)$, $R_{1}(m)$, $R_{2}(m)$, $R_{3}(m)$, $R_{4}(m)$, $R_{5}(m)$ be the number of representations of $m$ as a sum of four distinct elements of $\mathcal{A}$;
 four (not necessarily distinct) elements of $\mathcal{A}$; two elements of $\mathcal{A}$ with twice another (possibly the same) element of $\mathcal{A}$;
twice one element of $\mathcal{A}$ with twice another (possibly the same) element of $\mathcal{A}$;
one element of $\mathcal{A}$ with three times another (possibly the same) element of $\mathcal{A}$;
four times an element of $\mathcal{A}$, respectively.
Then \begin{equation}{\label{2.1}}R(m)=R_{1}(m)-6R_{2}(m)+3R_{3}(m)+8R_{4}(m)-6R_{5}(m).\end{equation}
\end{lem}
\begin{proof} Write
\begin{eqnarray*}
\mathcal{A}_{0}&=&\left\{(a_{1},a_{2},a_{3},a_{4})\in \mathcal{A}^{4}:\; m=a_{1}+a_{2}+a_{3}+a_{4},  \; a_{1},a_{2},a_{3},a_{4} \text{ all distinct}\right\},\\
\mathcal{A}_{1}&=&\left\{(a_{1},a_{2},a_{3},a_{4})\in \mathcal{A}^{4}:\; m=a_{1}+a_{2}+a_{3}+a_{4}\right\},\\
\mathcal{A}_{2}&=&\left\{(a_{1},a_{2},a_{3},a_{3})\in \mathcal{A}^{4}:\; m=a_{1}+a_{2}+2a_{3},\;  a_{1},a_{2},a_{3} \text{ possibly the same}\right\},\\
\mathcal{A}_{3}&=&\left\{(a_{1},a_{1},a_{2},a_{2})\in \mathcal{A}^{4}:\; m=2a_{1}+2a_{2},\;  a_{1},a_{2} \text{ possibly the same}\right\},\\
\mathcal{A}_{4}&=&\left\{(a_{1},a_{2},a_{2},a_{2})\in \mathcal{A}^{4}:\; m=a_{1}+3a_{2},\;  a_{1},a_{2} \text{ possibly the same}\right\},\\
\mathcal{A}_{5}&=&\left\{(a_{1},a_{1},a_{1},a_{1})\in \mathcal{A}^{4}:\; m=4a_{1}\right\}.
\end{eqnarray*}
To proof (\ref{2.1}), we consider the following cases:

{\bf Case 1.} $m$ cannot be represented as a sum of four elements of $\mathcal{A}$.
 Then $$R(m)=R_i(m)=0,\quad i=1,\ldots,5.$$

{\bf Case 2.} $m=a_{1}+a_{2}+a_{3}+a_{4}$, $a_{i}\in \mathcal{A}$ and $a_{i}^{'}s$ are distinct. Then
 the contribution $\left(a_{1},a_{2},a_{3},a_{4}\right)$ to $R(m)$ and $R_1(m)$ both are 24, and
the contribution $\left(a_{1},a_{2},a_{3},a_{4}\right)$ to $R_{2}(m)$, $R_{3}(m)$, $R_{4}(m)$, $R_{5}(m)$ are all zero.

{\bf Case 3.} $m=a_{1}+a_{2}+a_{3}+a_{3}$, $a_{i}\in \mathcal{A}$ and $a_{1},a_{2},a_{3}$ are distinct. Then
 the contribution $\left(a_{1},a_{2},a_{3},a_{3}\right)$ to $R(m)$, $R_{3}(m)$, $R_{4}(m)$, $R_{5}(m)$ are all zero,
the contribution $\left(a_{1},a_{2},a_{3},a_{3}\right)$ to $R_1(m)$ and $R_2(m)$ are 12, 2, respectively.

{\bf Case 4.} $m=a_{1}+a_{1}+a_{2}+a_{2}$, $a_{i}\in \mathcal{A}$ and $a_{1},a_{2}$ are distinct. Then
the contribution $\left(a_{1},a_{1},a_{2},a_{2}\right)$ to $R(m)$, $R_{4}(m)$, $R_{5}(m)$ are all zero,
the contribution $\left(a_{1},a_{1},a_{2},a_{2}\right)$ to $R_1(m)$, $R_2(m)$ and $R_3(m)$ are 6, 2, 2, respectively.

{\bf Case 5.} $m=a_{1}+a_{2}+a_{2}+a_{2}$, $a_{i}\in \mathcal{A}$ and $a_{1},a_{2}$ are distinct. Then
the contribution $\left(a_{1},a_{2},a_{2},a_{2}\right)$ to $R(m)$, $R_{3}(m)$, $R_{5}(m)$ are all zero,
the contribution $\left(a_{1},a_{2},a_{2},a_{2}\right)$ to $R_1(m)$, $R_2(m)$ and $R_4(m)$ are 4, 2, 1, respectively.

{\bf Case 6.} $m=a_{1}+a_{1}+a_{1}+a_{1}$, $a_{1}\in \mathcal{A}$. Then
the contribution $\left(a_{1},a_{1},a_{1},a_{1}\right)$ to $R_i(m)$, $i=1,\ldots,5$ are all one,
the contribution $\left(a_{1},a_{1},a_{2},a_{2}\right)$ to $R(m)$ is zero.

By the discussion of the above cases, we have (\ref{2.1}).

This completes the proof of Lemma \ref{lem2}.
\end{proof}

The following lemma is contained in the proof of Proposition 3.4 of paper \cite{Gallardo02}. For the readability of the paper, we rewrite the proof.
\begin{lem}{\label{lem20}}
Let $\mathcal{A}\subseteq \mathbb{Z}_n$ and $h$ be a positive integer. If $n\nmid h$, then $$\left|S\left(\frac{h}{n}\right)\right|\leq \frac{n}{3}.$$
\end{lem}

\begin{proof}Write $$q=\gcd(h, n), \;h=qh',\; n=qd.$$
Then $(h',d)=1$ and
\begin{equation}{\label{2.3}} S\left(\frac{h}{n}\right)=S\left(\frac{h'}{d}\right).\end{equation}
For any $j$, we have
$$\sharp\left\{a\in \mathcal{A}:\; h'a\equiv j\pmod d\right\}\leq \displaystyle\frac{n}{d}.$$
For $j=1, \ldots, d$, let
$$x_{j}=\sharp\left\{a\in \mathcal{A}:\; h'a\equiv j\pmod d\right\}.$$
Then by (\ref{2.3}) we have
$$\displaystyle S\left(\frac{h}{n}\right)=S\left(\frac{h'}{d}\right)=\sum\limits_{a\in \mathcal{A}}e\left(\frac{h'a}{d}\right)=\sum\limits_{j=1}^{d}x_{j}e\left(\frac{j}{d}\right).$$
Put $$\underline{x}=(x_{1},x_{2},\ldots, x_{d}).$$
Then by Lemma \ref{lem1}, we have
$$\displaystyle \left|S\left(\frac{h}{n}\right)\right|\leq \frac{n}{2d\sin\pi/2d}.$$
Noting that $\sin t\geq3t/\pi,\; 0\leq t\leq\pi/6$, we have
$$\left|S\left(\frac{h}{n}\right)\right|\leq \frac{n}{3}.$$

This completes the proof of Lemma \ref{lem20}.
\end{proof}

\section{\bf Proof of Theorem \ref{thm2}.}
{\bf Fact I} For any $\mathcal{A}\subset \mathbb{Z}_n$, we have $|3^{\wedge}\mathcal{A}|\geq |\mathcal{A}|-2$.

In fact, we may assume that
$\mathcal{A}=\{0=a_0, a_1, \ldots, a_{k-1}\}\subset \mathbb{Z}_n$.
Noting that $$a_0+a_1+a_{k-1}, \,a_0+a_2+a_{k-1},\,\cdots, a_{0}+a_{k-2}+a_{k-1} \in 3^{\wedge}\mathcal{A},$$ we have $|3^{\wedge}\mathcal{A}|\geq |\mathcal{A}|-2$.

Let $E$ and $O$ be the even and the odd elements of $\mathbb{Z}_n$, respectively. Write
$$\mathcal{A}_{e}=E\cap\mathcal{A}, \quad \mathcal{A}_{o}=O\cap\mathcal{A}.$$
Since $|\mathcal{A}|\geq n/2+3$, we have $\min\{|\mathcal{A}_{e}|,|\mathcal{A}_{o}|\}\geq3$.

If $x$ is even, then
$$2^{\wedge}\mathcal{A}_{e},\ x-2^{\wedge}\mathcal{A}_{o},\ 3^{\wedge}\mathcal{A}_{e}\subset E.$$
By Lemma \ref{lem4} and Fact I, we have
$$\displaystyle |2^{\wedge}\mathcal{A}_{e}|+|x-2^{\wedge}\mathcal{A}_{o}|\geq|\mathcal{A}_{e}|+|\mathcal{A}_{o}|=|\mathcal{A}|\geq\frac{n}{2}+3,$$
$$\displaystyle |3^{\wedge}\mathcal{A}_{e}|+|x-2^{\wedge}\mathcal{A}_{o}|\geq|\mathcal{A}_{e}|-2+|\mathcal{A}_{o}|\geq\frac{n}{2}+1.$$
Thus $$2^{\wedge}\mathcal{A}_{e}\cap (x-2^{\wedge}\mathcal{A}_{o})\neq\emptyset, \quad 3^{\wedge}\mathcal{A}_{e}\cap (x-2^{\wedge}\mathcal{A}_{o})\neq\emptyset.$$
That is,
$$ x\in 2^{\wedge}\mathcal{A}_{e}+2^{\wedge}\mathcal{A}_{o},\quad
 x\in 3^{\wedge}\mathcal{A}_{e}+2^{\wedge}\mathcal{A}_{o}.$$
Hence $E\subset4^{\wedge}\mathcal{A}$, $E\subset5^{\wedge}\mathcal{A}$.

If $x$ is odd, then
$$3^{\wedge}\mathcal{A}_{o},\ x-\mathcal{A}_{e},\ x-2^{\wedge}\mathcal{A}_{e}\subset O.$$
By Lemma \ref{lem4} and Fact I, we have
$$\displaystyle |3^{\wedge}\mathcal{A}_{o}|+|x-\mathcal{A}_{e}|\geq|\mathcal{A}_{o}|-2+|\mathcal{A}_{e}|\geq\frac{n}{2}+1,$$
$$\displaystyle |3^{\wedge}\mathcal{A}_{o}|+|x-2^{\wedge}\mathcal{A}_{e}|\geq|\mathcal{A}_{o}|-2+|\mathcal{A}_{e}|\geq\frac{n}{2}+1.$$
Thus $$3^{\wedge}\mathcal{A}_{o}\cap (x-\mathcal{A}_{e})\neq\emptyset, \quad 3^{\wedge}\mathcal{A}_{o}\cap (x-2^{\wedge}\mathcal{A}_{e})\neq\emptyset.$$
That is $$x\in 3^{\wedge}\mathcal{A}_{o}+\mathcal{A}_{e},\quad
 x\in 3^{\wedge}\mathcal{A}_{0}+2^{\wedge}\mathcal{A}_{e}.$$
Hence $O\subset4^{\wedge}\mathcal{A}$, $O\subset5^{\wedge}\mathcal{A}$.

This completes the proof of Theorem \ref{thm2}.

\section{\bf Proof of Theorem \ref{thm3}.}
By Lemma \ref{lem5} and Lemma \ref{lem6}, we have
$$\displaystyle|\mathcal{A}_{e}|\geq\frac{n}{4}+3>\frac{|E|+2}{2},$$ thus we have $2^{\wedge}\mathcal{A}_{e}=E.$

If $x$ is even, then
$$2^{\wedge}\mathcal{A}_{e},\ x-2^{\wedge}\mathcal{A}_{o}\subset E,$$
and $$|2^{\wedge}\mathcal{A}_{e}|+|x-2^{\wedge}\mathcal{A}_{o}|\geq|E|+|\mathcal{A}_{o}|-1\geq\frac{n}{2}+1.$$
Thus $$2^{\wedge}\mathcal{A}_{e}\cap (x-2^{\wedge}\mathcal{A}_{o})\neq\emptyset.$$
That is, $x\in 2^{\wedge}\mathcal{A}_{e}+2^{\wedge}\mathcal{A}_{o}.$
Hence $E\subset4^{\wedge}\mathcal{A}$.

If $x$ is odd, then
$$3^{\wedge}\mathcal{A}_{e},\ x-\mathcal{A}_{o}\subset E.$$
Take $b\in\mathcal{A}_{e}$ and consider  $\mathcal{B}=\mathcal{A}_{e}\setminus\{b\}$. Since
$$\displaystyle|\mathcal{B}|\geq\frac{n}{4}+2>\frac{|E|+2}{2}.$$ Again by Lemma \ref{lem5}, we have $2^{\wedge}\mathcal{B}=E$ and then $b+2^{\wedge}\mathcal{B}=E$.
Hence $E\subset3^{\wedge}\mathcal{A}_{e}$.

Noting that $$\displaystyle |3^{\wedge}\mathcal{A}_{e}|+|x-\mathcal{A}_{o}|\geq|E|+2=\frac{n}{2}+2,$$
we have $3^{\wedge}\mathcal{A}_{e}\cap (x-\mathcal{A}_{o})\neq\emptyset.$ That is, $x\in 3^{\wedge}\mathcal{A}_{e}+\mathcal{A}_{o}.$
Hence $O\subset4^{\wedge}\mathcal{A}$.

This completes the proof of Theorem \ref{thm3}.

\section{\bf Proof of Theorem \ref{thm1}.}
Write $|\mathcal{A}|=k$. Let $R(m)$, $R_i(m)(i=1,\ldots,5)$ be as in Lemma \ref{lem2}.
Noting that\begin{displaymath}
\displaystyle\frac{1}{n}\sum\limits_{h=1}^{n}e\left(\frac{ht}{n}\right)=
\begin{cases}
1, \quad&\; \text{if } n\mid t,\\
0,\quad&\; \text{otherwise},
\end{cases}
\end{displaymath}
for any $m\in\mathbb{Z}_{n}$, we have
\begin{eqnarray}{\label{2.2}}
R_{1}(m)&=&\sum\limits_{\substack{a_{1},a_{2},a_{3},a_{4}\in \mathcal{A}\\ m=a_{1}+a_{2}+a_{3}+a_{4}}}1\nonumber\\
&=&\displaystyle\sum\limits_{a_{1}\in\mathcal{A}}\sum\limits_{a_{2}\in\mathcal{A}}\sum\limits_{a_{3}\in\mathcal{A}}
\sum\limits_{a_{4}\in\mathcal{A}}\frac{1}{n}\sum\limits_{h=1}^{n}e\left(\frac{h (-m+a_{1}+a_{2}+a_{3}+a_{4})}{n}\right)\nonumber\\
&=&\displaystyle\frac{1}{n}\sum\limits_{h=1}^{n}S\left(\frac{h}{n}\right)^{4}e\left(-\frac{hm}{n}\right)\nonumber\\
&\geq&\displaystyle\frac{k^{4}}{n}-\frac{1}{n}\sum\limits_{h=1}^{n-1}\left|S\left(\frac{h}{n}\right)\right|^{4}\\
&\geq&\displaystyle\frac{k^{4}}{n}-\left(\max\limits_{h:n\nmid h}\left|S\left(\frac{h}{n}\right)\right|\right)^{2}\left(\frac{1}{n}\sum\limits_{h=1}^{n-1}\left|S\left(\frac{h}{n}\right)\right|^{2}\right)\nonumber\\
&\geq&\displaystyle\frac{k^{4}}{n}-\left(\max\limits_{h:n\nmid h}\left|S\left(\frac{h}{n}\right)\right| \right)^{2} \left(\frac{1}{n}\sum\limits_{h=1}^{n}\left|S\left(\frac{h}{n}\right)\right|^{2}-\frac{k^{2}}{n}\right)\nonumber\\
&=&\displaystyle\frac{k^{4}}{n}-\left(\max\limits_{h:n\nmid h}\left|S\left(\frac{h}{n}\right)\right|\right)^{2}\left(k-\frac{k^{2}}{n}\right).\nonumber
\end{eqnarray}
By Lemma \ref{lem20}, we have
\begin{equation}{\label{3.2}}R_{1}(m)\geq\displaystyle\frac{k^{4}}{n}-\left(\frac{n}{3}\right)^{2}\left(k-\frac{k^{2}}{n}\right).\end{equation}

If $k>\alpha n$ and $n>\displaystyle\frac{54}{9\alpha^{3}+\alpha-1}$, then
$$9k^3-n^3+kn^2>n^3(9\alpha^3+\alpha-1)>54n^2>54kn,$$
we have \begin{equation}{\label{3.3}}\frac{k^{3}}{n}-\frac{n^2}{9}+\frac{kn}{9}-6k>0.\end{equation}
Moreover, $R_{2}(m)\leq k(k-1)$, $R_{5}(m)\leq k$. By Lemma \ref{lem2}, (\ref{2.2})-(\ref{3.3}), we have
\begin{eqnarray*}R(m)&=&R_{1}(m)-6R_{2}(m)+3R_{3}(m)+8R_{4}(m)-6R_{5}(m)\\
&\geq&\displaystyle\frac{k^{4}}{n}-\left(\frac{n}{3}\right)^{2}\left(k-\frac{k^{2}}{n}\right)-6k(k-1)-6k>0.\end{eqnarray*}

This completes the proof of Theorem \ref{thm1}.

\section{\bf Some Remark}

We can show that the answer to Problem \ref{prob} for $|\mathcal{A}|=4,5$ is positive.

For $|\mathcal{A}|=4$, it is easy to see that $|3^{\wedge}\mathcal{A}|=4$.

For $|\mathcal{A}|=5$, we may
assume that $$\mathcal{A}=\{0=a_0,a_1,a_2,a_3,a_4\}\subset\mathbb{Z}_{n}.$$
It is obviously that $$a_0+a_1+a_2,\,a_0+a_1+a_3,\,a_0+a_1+a_4\in 3^{\wedge}\mathcal{A},$$
$$a_1+a_2+a_3\neq a_0+a_1+a_2, \quad a_0+a_1+a_3,$$  $$a_1+a_2+a_4\neq a_0+a_1+a_2, \quad a_0+a_1+a_4.$$
If $a_2+a_3=a_1+a_4,\, a_1+a_3=a_2+a_4$, then
$$a_1+a_2+a_3=a_1+a_1+a_4, \quad a_1+a_2+a_4=a_1+a_1+a_3.$$
Thus $a_1+a_2+a_3\neq a_0+a_1+a_4,\quad
a_1+a_2+a_4\neq a_0+a_1+a_3.$
Hence $$a_1+a_2+a_3,\,a_1+a_2+a_4\in 3^{\wedge}\mathcal{A}\setminus\{a_0+a_1+a_2,a_0+a_1+a_3,a_0+a_1+a_4\}.$$
If $a_2+a_3=a_1+a_4,\, a_1+a_3\neq a_2+a_4$, then
$$a_1+a_2+a_3,\,a_0+a_2+a_4\in 3^{\wedge}\mathcal{A}\setminus\{a_0+a_1+a_2,a_0+a_1+a_3,a_0+a_1+a_4\}.$$
If $a_2+a_3\neq a_1+a_4,\, a_1+a_3= a_2+a_4$, then
$$a_0+a_2+a_3,\,a_1+a_2+a_4\in 3^{\wedge}\mathcal{A}\setminus\{a_0+a_1+a_2,a_0+a_1+a_3,a_0+a_1+a_4\}.$$
If $a_2+a_3\neq a_1+a_4,\, a_1+a_3\neq a_2+a_4$, then
$$a_0+a_2+a_3,\,a_0+a_2+a_4\in 3^{\wedge}\mathcal{A}\setminus\{a_0+a_1+a_2,a_0+a_1+a_3,a_0+a_1+a_4\}.$$
Hence $|3^{\wedge}\mathcal{A}|\geq 5$.

\end{document}